\documentclass[sigplan,screen]{acmart}

\newcommand{\N}{\mathbb N}
\newcommand{\R}{\mathbb{R}}
\newcommand{\CC}{\mathbb{C}}
\newcommand{\Z}{\mathbb{Z}}

\newcommand{\proj}{\mathbb{P}}
\newcommand{\GL}{\mathrm{GL}}

\newcommand{\E}{\mathbb{E}\,}
\newcommand{\vol}{\mathrm{vol}}
\newcommand{\cM}{\mathcal{M}}
\newcommand{\s}{\sigma}
\newcommand{\e}{\varepsilon}

\newcommand{\diag}{\mathrm{diag}} 
\newcommand{\inte}{\mathrm{int}} 
\newcommand{\cone}{\mathrm{cone}} 
\newcommand{\conv}{\mathrm{conv}} 
 
\newcommand{\Vt}{\mathrm{Vert}} 
\newcommand{\Gr}{\mathrm{Gr}} 
\newcommand{\cL}{\mathcal{L}} 
\newcommand{\caM}{\mathcal{M}}

\newtheorem{thm}{Theorem}
\newtheorem{lemma}[thm]{Lemma}

\newtheorem{prop}[thm]{Proposition}

\newtheorem{defi}[thm]{Definition}
\newtheorem{conj}{Conjecture}

\theoremstyle{remark} 
\newtheorem{remark}[thm]{Remark}
\newtheorem{example}[thm]{Example}

\numberwithin{equation}{section}
\numberwithin{thm}{section}

\AtBeginDocument{%
  }

\copyrightyear{2023} 
\acmYear{2023} 
\setcopyright{acmlicensed}\acmConference[ISSAC 2023]{International Symposium on Symbolic and Algebraic Computation 2023}{July 24--27, 2023}{Tromsø, Norway}
\acmBooktitle{International Symposium on Symbolic and Algebraic Computation 2023 (ISSAC 2023), July 24--27, 2023, Tromsø, Norway}
\acmPrice{15.00}
\acmDOI{10.1145/3597066.3597105}
\acmISBN{979-8-4007-0039-2/23/07}
\begin{document}

%%
%% The "title" command has an optional parameter,
%% allowing the author to define a "short title" to be used in page headers.
\title{Real zeros of mixed random fewnomial systems}

%%
%% The "author" command and its associated commands are used to define
%% the authors and their affiliations.
%% Of note is the shared affiliation of the first two authors, and the
%% "authornote" and "authornotemark" commands
%% used to denote shared contribution to the research.

\author{Peter B\"urgisser}
\affiliation{%
  \institution{Institute of Mathematics, Technische Universit\"at Berlin}
  \streetaddress{Strasse des 17. Juni 136}
  \city{Berlin}
  \postcode{10623}
  \country{Germany}} 
\email{pbuerg@math.tu-berlin.de} 

%%
%% By default, the full list of authors will be used in the page
%% headers. Often, this list is too long, and will overlap
%% other information printed in the page headers. This command allows
%% the author to define a more concise list
%% of authors' names for this purpose.
\renewcommand{\shortauthors}{B\"urgisser}

%%
%% The abstract is a short summary of the work to be presented in the
%% article.
\begin{abstract}
Consider a system $f_1(x)=0,\ldots,f_n(x)=0$ 
of $n$ random real polynomials in $n$ variables,
where each $f_i$ has a prescribed set of exponent vectors in a set
$A_i \subseteq \mathbb{Z}^n$ of cardinality~$t_i$, whose convex hull is denoted $P_i$. 
Assuming that the coefficients of the $f_i$ are independent standard Gaussian, 
we prove that the expected number of zeros of the random system in the positive orthant is 
at most 
$(2\pi)^{-\frac{n}{2}} \, V_0\, (t_1-1)\ldots (t_n-1)$. 
Here $V_0$ denotes the number of vertices of the Minkowski sum $P_1+\ldots + P_n$.
However, this bound does not improve over the bound in \cite{BETC:19} for the unmixed case,
where all supports $A_i$ are equal.
All arguments equally work for real exponent vectors. 
\end{abstract}

%%
%% The code below is generated by the tool at http://dl.acm.org/ccs.cfm.
%% Please copy and paste the code instead of the example below.
%%
\begin{CCSXML}
<ccs2012>
<concept>
<concept_id>10003752.10010061.10010063</concept_id>
<concept_desc>Theory of computation~Computational geometry</concept_desc>
<concept_significance>300</concept_significance>
</concept>
<concept>
<concept_id>10002950.10003741.10003732.10003735</concept_id>
<concept_desc>Mathematics of computing~Integral calculus</concept_desc>
<concept_significance>300</concept_significance>
</concept>
</ccs2012>
\end{CCSXML}

\ccsdesc[300]{Theory of computation~Computational geometry}
\ccsdesc[300]{Mathematics of computing~Integral calculus}

%%
%% Keywords. The author(s) should pick words that accurately describe
%% the work being presented. Separate the keywords with commas.
\keywords{fewnomials, random polynomials, real algebraic geometry, sparsity} 
%% A "teaser" image appears between the author and affiliation
%% information and the body of the document, and typically spans the
%% page.
%\begin{teaserfigure}
%  \includegraphics[width=\textwidth]{sampleteaser}
% \caption{Seattle Mariners at Spring Training, 2010.}
%  \Description{Enjoying the baseball game from the third-base
%  seats. Ichiro Suzuki preparing to bat.}
%  \label{fig:teaser}
%\end{teaserfigure}

%\received{20 February 2007}
%\received[revised]{12 March 2009}
%\received[accepted]{5 June 2009}

%%
%% This command processes the author and affiliation and title
%% information and builds the first part of the formatted document.
\maketitle

%\tableofcontents {}

\section{Introduction}

In many applications, we want to understand or find the positive real
solutions of a system of multivariate polynomial equations, e.g., see
\cite{drton-sturmfels:09,horn-jackson:72,sottile-book:11}. 
Bezout's theorem, which bounds the number complex zeros in terms of degrees, 
usually highly overestimates the number of real zeros. This can be already
seen from Descartes' rule of signs~\cite[p.~42]{descartes}, which 
implies that a real univariate polynomial with $t$ terms 
has at most $t-1$ positive zeros. 
In 1980, Khovanskii~\cite{kho:80} 
obtained a far reaching generalization of Descartes' rule. 
He showed that  the number of nondegenerate\footnote{i.e., the Jacobian of the system does not vanish at the zero.}
positive solutions
of a system $f_1(x)=0,\ldots,f_n(x)=0$ 
of $n$ real polynomial equations in~$n$ variables is bounded only in terms of~$n$
and the number $t$ of distinct exponent vectors occurring in the system.
This result in fact allows for any real exponents.
Following Kushnirenko, one speaks of {\em fewnomial systems},
with the idea that the number $t$ of terms is small, see~\cite{kho:91}. 

Understanding the complex zeros of fewnomial systems 
is much simpler: the famous BKK-Theorem~\cite{bernstein:75,kushnirenko:76} 
states that for given finite supports $A_1,\ldots,A_n\subseteq\Z^n$ 
and Laurent polynomials $f_i(x) =\sum_{a\in A} c_i(a) x_1^{a_1}\cdots x_n^{a_n} $ 
with generic complex coefficients $c_i(a)$, 
the number of complex solutions in $(\CC^\times)^n$ of 
a corresponding system $f_1(x)=0,\ldots,f_n(x)=0$ 
is given by $n!$ times the mixed volume of the Newton polytopes
$P_1,\ldots,P_n$, where $P_i$ is defined as the convex hull of~$A_i$. 

Note that the number of real zeros has little to do with the metric properties of $P_i$:
indeed, replacing $A_i$ by a nonzero multiple $m_iA_i$ 
amounts to substituting $x_i$ by $x_i^{m_i}$. %($m>0$). 
Clearly, this does not change the number of 
positive real zeros of a fewnomial system, however $P_i$ has been 
replaced by $m_i P_i$. 

The bound on the number of real zeros obtained by Khovanskii 
is exponential in the number~$t$. %of exponent vectors. 
It is widely conjectured that this bound is far from optimal: 
in fact it is conjectured~\cite{rojas_phillipson} that for fixed~$n$, 
the number of nondegenerate positive solutions 
of a fewnomial system with $t$ exponent vectors is bounded by a polynomial in $t$. 
Quite surprisingly, this question is open even for $n=2$! 
For results in special cases, we refer to 
\cite{biha_so:07,avendano,sottile-book:11,ko-po-ta:15,KPS:15,paul}.
Moreover, there is a very interesting connection to complexity theory \cite{koiran,bu-bri:18}. 

Given this state of affairs of real fewnomial theory, a possible way to advance is to ask 
what happens in generic situations. This can be made formal by considering {\em random} 
real fewnomial systems, see \cite{Bez2,edel_kost,rojas_avg,malajovich,BETC:19,malajovich:22}. 
Fix supports $A_1,\ldots,A_n\subseteq\Z^n$ of cardinality~$t_1,\ldots,t_n$, respectively, 
and consider a system of $n$~random polynomials $f_i(x)$ as above, 
but now the coefficients $c_i(a)$ are assumed to be independent standard Gaussian. 
Let us denote by $\E(A_1,\ldots,A_n)$ the expectation of the number of 
nondegenerate positive real zeros of such system. 
Actually, we work in more generality, allowing any subsets $A_i$ of $\R^n$;  
see Section~\ref{se:mixed}. 
%There is a large literature on the real zeros of random polynomials: 
%we refer to~\cite{BETC:19} for references.
%\cite{rojas_avg,malajovich}, \cite{malajovich:22} 
%\cite{Bez2}

%e.g., see Edelman and Kostlan~\cite{edel_kost}.
%The investigation of the real zeros of systems of random multivariate polynomials 
%has been given a lot of impetus by the discovery of a 
%real probabilistic version of B\'ezout's theorem by Shub and Smale~\cite{Bez2} 
%and Kostlan~\cite{kostlan93}.

%Probability theory and integral geometry allows one to make statements about 
%the number of real zeros of random polynomials systems. 
%The work by Shub and Smale~\cite{Bez2} and Kostlan~\cite{kostlan93}
%on real probabilistic version of B\'ezout's theorem 
%provided a lot of motivation. 

In~\cite{BETC:19} it was proven that 
$\E(A,\ldots,A) \le 2^{1-n}\binom{t}{n}$. 
The main result of the present paper 
is an extension of this to the mixed case, 
where the fewnomials may have have different supports~$A_i$. 
Our bound depends on the combinatorial structure of the 
Minkowski sum $P_1+\ldots + P_n$ through the number of its vertices. 
We remark that our proof is quite different from the one in~\cite{BETC:19}, 
which is rather indirect. 
Clearly, the number $V_0(P_i)$ of vertices of $P_i$ is at most~$t_i$.
Moreover, $V_0(P_1+\ldots + P_n)\le V_0(P_1)\cdots V_0(P_n)$ 
and this bound is known to be sharp~\cite{fukuda-weibel:07}.

%and any system of variances $\s\colon A \to \Rp$. 

% and left it as a challenge to find a more direct analytical proof. 
%In this paper, we achieve both of these goals.

\begin{thm}\label{th:main-mixed-real}
If the $A_i\subseteq\R^n$ are finite nonempty sets of cardinality~$t_i$
and with convex hull~$P_i$, 
%If $P_i$ denotes the convex hull of the finite nonempty set $A_i\subseteq\R^n$ of cardinality~$t_i$, 
for $i=1,\ldots,n$, then 
$$
 \E(A_1,\ldots,A_n) \ \le\ 
  (2\pi)^{-\frac{n}{2}} \, V_0\, (t_1-1)\ldots (t_n-1) . 
$$
Here $V_0$ denotes the number of vertices of the Minkowski sum $P:=P_1+\ldots + P_n$.
\end{thm}

The bound in this theorem looks similar to the one in a conjecture attributed to Kushnirenko,
which states that the number of positive nondegenerate zeros 
is always bounded by $(t_1-1)\cdots (t_n-1)$. However, this was disproved 
in~\cite{haas:02}, already in the special case $n=2$.\footnote{This conjecture was never published by Kushnirenko and apparently, he did not believe in it.}

In the unmixed situation, where all supports equal~$A$, it is well known~\cite{edel_kost} 
%Edelman and Kostlan~\cite{edel_kost} 
that the expected number of positive zeros can be expressed by the 
volume of the image of the Veronese like map 
$\R^n_{>0} \to \proj(\R^A)$ sending $x$ to $[x^a]_{a\in A}$. 
This is a consequence of the kinematic formula for real projective spaces. 
In the mixed situation, there is no such simple characterization:  
we work with the more complicated kinematic formula for products of 
projective spaces (Theorem~\ref{th:KF}) that we derive from~\cite{howard:93,BL:19}.  
%To analyze this, we 
After passing to exponential coordinates $w=\log x$, we bound 
the resulting integral over~$\R^n$ with a strategy 
inspired by the theory of toric varieties.
The normal fan of the polytope $P$ affords 
a decomposition of $\R^n$ into the normal cones $C$ at the vertices of~$P$.
The resulting integral over $C$ can be bounded in terms of the characteristic function of 
the dual cone of~$C$. 
Finally, an explicit a priori bound on this characteristic function (Proposition~\ref{pro:crucial-trick}) 
completes the argument.

\subsection{The univariate case and a conjecture}

The univariate case ($n=1$) was settled,
up to multiplicative constants,
by Jindal et al.~\cite{jindal_et_al:20}. 
They showed that for any subset 
$S\subseteq\R$ of cardinality~$t$, we have 
\begin{equation}\label{eq:jindal}
 \mbox{$\E(S) \le  \frac{2}{\pi} \sqrt{t-1}$} .
\end{equation}
Moreover,  they constructed a sequence $S_t\subseteq\Z$ of supports of cardinality~$t$ 
with $\E(S_t) \ge c\sqrt{t}$ for some constant $c>0$. 
Consider for $t_1,\ldots,t_n\ge 1$ the supports 
$A_{1}:= S_{t_1} \times 0\ldots\times 0,\ldots, %$A_{2}:= 0\times S_{t_2} \times 0\ldots\times 0$, $\ldots$, 
A_{n}:= 0\times \ldots \times 0 \times S_{t_n} $. 
These supports describe a system of $n$ equations, where the $i$th equation depends on $x_i$ only. 
Therefore, 
$\E(A_{1},\ldots,A_{n}) = \E(S_{t_1})\cdots \E(S_{t_n})$, 
which with the above leads to the lower bound 
\begin{equation}\label{eq:LB}
  \E(A_{1},\ldots,A_{n}) \ \ge\ c^n \sqrt{t_1\cdots t_n} .
\end{equation}
We complement this by showing that for any $A= S_1\times\ldots\times S_n$ in product form, 
the expectation $\E(A,\ldots,A)$ can be expressed in terms of the $\E(S_i)$ as follows. 

\begin{prop}\label{pro:MVR}
If $A= S_1\times\ldots\times S_n$ for finite $S_i\subseteq\R$, then 
$$
 \E(A,\ldots,A) = \pi^n (\vol(\proj^n))^{-1} \, \E(S_1)\cdots \E(S_n) .
%\frac{\pi^n}{\vol(\proj^n)} \, \E(S_1)\cdots \E(S_n) .
$$
\end{prop}

We conjecture that the lower bound~\eqref{eq:LB} is optimal in the following sense. 

\begin{conj}\label{conj:N}
Let $A_i\subseteq\R^n$ be  finite nonempty sets of cardinality~$t_i$ with convex hull~$P_i$,
for $i=1,\ldots,n$. 
We denote by $V_0$ the number of vertices of $P_1+\ldots + P_n$.
Then 
$$
 \E(A_1,\ldots,A_n) \ \le \kappa(n,V_0) \sqrt{t_1 \cdots t_n} 
$$
for some function $\kappa:\N^2\to\N$. In particular,
for $A\subseteq\R^n$ of cardinality~$t$, we have 
$\E(A,\ldots,A) \ \le \kappa(n,V_0) \, t^{\frac{n}{2}}$. 
\end{conj}

%Essentially, we believe that $t_i$ in Theorem~\ref{th:main-mixed-real} can be 
%replaced by its square root. 

In the special case $A= S_1\times\ldots\times S_n$, 
%of cardinality $t$, 
by combining ~\eqref {eq:jindal} with Proposition~\ref{pro:MVR}, 
we obtain 
$\E(A,\ldots,A)\vol(\proj^n) \ \le\  2^n \sqrt{t}$  
with $t=\#A$, 
which is %much 
smaller than what Conjecture~\ref{conj:N} predicts.
%The bound~\eqref {eq:jindal} implies for such $A$ of cardinality~$t$ that 
%for some $c>0$ 
%Note that in this case, $N(A,\ldots,A)$ is 

\subsection{Improvement in unmixed case}

We can exponentially improve the dependence on $n$ in the bound of 
Theorem~\ref{th:main-mixed-real}
in the case where all supports are equal.
(Note $\vol(\proj^n)^{-1} = \Gamma(\frac{n+1}{2}) \pi^{-\frac{n+1}{2}}$.) 

\begin{prop}\label{th:main-unmixed-real}
For $A\subseteq\R^n$ of cardinality~$t\ge 1$ with convex hull $P$ and $V_0$ vertices, we have 
$$
 \E(A,\ldots,A) \ \le\ \frac{1}{\vol(\proj^n)} \, V_0\, {t-1 \choose n} .
%\frac{1}{\vol(\proj^n)} \, V_0\, {t-1 \choose n} .
$$
\end{prop}

Unfortunately, this bound 
has exponentially worse dependence on~$n$ than the bound 
$\E(A,\ldots,A) \le 2^{1-n} {t \choose n}$ in~\cite{BETC:19}.
For instance, for $t=n+k$ with fixed~$k$, 
$\E(A,\ldots,A)$ goes to $0$ exponentially fast as  $n\to\infty$ by~\cite{BETC:19}, 
so the system has no nondegenerate zero with overwhelming probability.
The bound in Proposition~\ref{th:main-unmixed-real} is too weak to reveal this!

%To compare this with the bound 
%$\E(A,\ldots,A) \le 2^{1-n} {t \choose n}$ 
%from~\cite{BETC:19}, note that
%$\vol(\proj^n)^{-1} = \Gamma(\frac{n+1}{2}) \pi^{-\frac{n+1}{2}}$.  
%{\tt CHECK! CASE $n+k$ with $k$ fixed?}
%Hence, for $n\to\infty$, the new bound goes asymptotically faster to zero than the old one.

\begin{remark}
The bound in~\cite{BETC:19} also holds for {\em nonstandard} centered Gaussian 
coefficients $c(a) \sim N(0,\s(a)^2)$. 
In this %more general 
situation, our proof of Theorem~\ref{th:main-unmixed-real} 
only leads to an upper bound with the additional factor 
$\big(\max_a \s(a)/\min_a \s(a)\big)^n$
(similarly for Theorem~\ref{th:main-mixed-real}).  
%We leave it as a challenge to remove this dependency.
\end{remark}

%any system of variances$\s(a)^2$, where  

%In \cite{BETC:19} we obtained in the case $n=1$ the better bound 
%$\E N(f) \le 2\pi^{-1} \sqrt{t}\log t$. %, provided $\s(a)=1$. %instead of  $O(t)$
%This bound was improved to $\E N(f) =O(\sqrt{t})$ in~\cite{jindal_et_al:20}.
%Moreover, optimality was show there: for each $t$ there exists $A\subseteq\Z$ of cardinality~$t$
%such that $\E N(f) =\Omega(\sqrt{t})$.

\subsection{Location of zeros}

We finish with a result on the typical location of the zeros. 
It is well known that for certain random real polynomials, 
the positive reals zeros $x$ tend to accumulate around~$1$:
see~\cite{edel_kost} for the dense 
and~\cite{jindal_et_al:20} for the sparse case. 
This means that $w=\log x$ accumulates around~$0$.
We generalize this to multivariate systems as follows.

\begin{thm}\label{th:real_concentrate_roots}
Fix a finite supports $A_1\ldots,A_n\subseteq\R^n$ and consider a random 
system~\eqref{eq:def_F_Syst} with independent standard Gaussian coefficients~$c_i(a)$
for the {\em stretched supports} $mA_i$, where $m\in \Z_{>0}$. Fix $\e>0$.
Then the probability that the system has a zero $w\in\R^n$ with $\|w\|> \e$
goes to zero, as $m\to\infty$.
\end{thm}

There are sophisticated results on the distributions of 
{\em complex zeros} of random fewnomials systems~\cite{shiff-zell:04,shiff-zell:11}.

%\subsection{Organization of paper}
%{\tt To be written.} 

%\subsection{Acknowledgments}

%To the best of our knowledge, the only previous results on random real fewnomial systems 
%are by Malajovich and Rojas~\cite{malajovich,rojas_avg}.
%These works provide upper bounds on the expected number of real zeros in terms of the (mixed) volume 
%of the Newton polytopes of the $f_i$. Thus these bounds depend also on the degree, while our bound 
%depends solely on the number of exponent vectors. 
%We refer to Shiffman and Zelditch~\cite{shiff-zell:04,shiff-zell:11} for results on the distribution of the zeros 
%of complex random fewnomials.  

\section{Preliminaries}

%We develop some auxiliary results that are needed for the proof of the main results. 
%%In Subsection~\ref{se:NF} 
%We first collect some facts on the vertices and the normal fan of a Minkowski sum of polytopes. 
%Subsection~\ref{se:ASF} contains a brief discussion of the average scaling factor from~\cite{BL:19}.
%Finally, we prove in Subsection~\ref{se:charf} an a priori upper bound on the characteristic 
%functions of convex cones, which is a key ingredient in the proof of Theorem~\ref{th:main-mixed-real}. 

\subsection{Metric properties of charts of projective space}\label{se:Pproj}

Consider the real projective space $\proj^m$.  
We shall identify the tangent space $T_{[y]} \proj^m$ 
at a point $[y]:=[y_0:\ldots : y_m]$ 
with $\R y^\perp$.  
The standard Riemannian metric on $\proj^m$ is defined by 
$\langle v,w \rangle_{[y]} := \|y\|^{-2} \langle v,w \rangle$ 
for $v,w \in \R y^\perp$. 
We denote by $P_y$ the orthogonal projection onto  $\R y^\perp$. 

Consider the affine chart 
$(\proj^m)_{y_0\ne 0}\to \R^m$, which maps 
$[y_0:\ldots : y_m]$ to $y_0^{-1} (y_1,\ldots,y_m)$. 
Its inverse is given by 
$$
 \pi\colon \R^m \to (\proj^m)_{y_0\ne 0},\, (y_1,\ldots,y_m) \mapsto [1:y_1:\ldots : y_m] .
$$
By \cite[Lemma 14.8]{Condition}, the derivative of $\pi$ at 
$y':=(y_1,\ldots,y_m)$ satisfies
$D_{y'}\pi = \|\pi(y')\|^{-1} P_y$, 
and therefore,
%From this we conclude for the spectral norm 
\begin{equation}\label{eq:der-pi}
 \|D_{y'}\pi \| \le \|\pi(y')\|^{-1} \le 1 .
\end{equation}

%since $P_y$ is an orthogonal projection.

\subsection{On the quantity $\sigma$}\label{se:ASF}

The relative position of two subspaces of a Euclidean vector space~$E$ can be quantified by a volume-like quantity,
which is crucial in the study of integral geometry in homogeneous spaces; 
see \cite{howard:93} and \cite[\S3.3]{BL:19}. 
To define this quantity, note first that there is an induced inner product on the exterior algebra $\Lambda(E)$ 
given by \cite[(2.1)]{BL:19} 
$$
 \langle v_1 \wedge\cdots\wedge v_k , w_1 \wedge\cdots\wedge w_k\rangle = \det (\langle v_i, w_j \rangle)_{1\le i,j\le k} .
$$
More concretely, 
$\|v_1\wedge\ldots\wedge v_n\| = |\det[v_1,\ldots,v_n]|$, 
where $[v_1,\ldots,v_n]$ denotes the matrix 
with columns~$v_i\in E=\R^n$

Let $V,W$ be linear subspaces of $E$ of complementary dimensions. 
%such that $k+m=n$. 
We define \cite[(3.3)]{BL:19}
\begin{equation} \label{eq:def-sigma}
 \s(V,W) :=\|v_1\wedge\ldots \wedge v_k\wedge w_1\wedge\ldots\wedge w_m \| \in [0,1] ,
\end{equation}
where $v_1,\ldots,v_k$ and $w_1,\ldots,w_m$ are orthonormal bases 
of $V$ and $W$, respectively. 
Clearly, $\s(V,W)=\s(W,V)$. 
Here are the extreme cases: 
$\s(V,W)=0$ iff $V\cap W\ne 0$ and 
$\s(V,W)=1$ iff $v$ and $W$ are orthogonal.  
%(This is clearly independent of the choice of the orthonormal bases.) 
We refer to Appendix~\ref{se:A0} for the proof of the following easy observation. 

\begin{prop}\label{prop:sigma-symm}
We have $\sigma(V^\perp,W^\perp) = |\det p\,|$, if the map 
$p\colon V^\perp \to W$ denotes the restriction of the 
orthogonal projection $E\to  W$ to $V^\perp$. Moreover,
$\sigma(V,W) = \sigma(V^\perp,W^\perp)$.
\end{prop}

Clearly, the definition \eqref{eq:def-sigma} can be extended to more than two subspaces;  
see \cite[(3.5)]{BL:19}. But if $W=W_1\oplus\ldots\oplus W_n$ is an orthogonal 
decomposition, we can reduce to the case of two subspace~\cite[Lemma A.6]{BL:19}. 
\begin{equation}\label{eq:reduce-sigma}
 \sigma(V,W_1,\ldots,W_n) = \sigma(V,W_1+\ldots+W_n) .
\end{equation}
%This is a consequence of \cite[Lemma A.6]{BL:19}.

\subsection{Characteristic functions of convex cones}\label{se:charf}

We prove here an priori upper bound on the characteristic function of a convex cone, 
which is a key ingredient in the proof of Theorem~\ref{th:main-mixed-real}. 

A convex cone $C\subseteq\R^n$ is called {\em proper} if it is $n$-dimensional 
and pointed, i.e., full-dimensional and contained in a half\-space. 
It is well known that a convex $C\subseteq \R^n$ is proper iff its {\em dual cone}
$$
 C^* := \{ x\in \R^n \mid \forall y\in C\ \langle x,y \rangle \ge 0 \}
$$
is proper. 
Let $g\in\GL(n,\R)$. Then $K:=g(C)$ is a proper cone and 
$g^T(K^*) = C^*$. 
We denote by $\inte(C)$ the interior of $C$. 

We assign to a proper cone $C\subseteq\R^n$ the function
%v_C\colon \inte(C) \to\R_{++} on the interior $C^*$ of $C$
\begin{equation}\label{eq:char-function}
 v_C\colon \inte(C^*) \to\R_{>0},\ 
 v_C(x) := \int_{C} e^{-\langle x,y\rangle}\, dy .  
\end{equation}
One calls $v_C$ the  {\em characteristic function} 
(or Koszul-Vinberg characteristic) of $C^*$. 
It is a useful analytic tool for investigating convex cones, e.g., 
see \cite[I.3]{faraut_koranyi:94} and \cite{guler:96}. 
%\begin{example}\label{ex:v-pos-orthant}
E.g., $\R^n_{>0}$ is self dual and 
$v_{\R^n_{>0}}(x) = (x_1\cdot\ldots\cdot x_n)^{-1}$
for $x\in \R^n_{>0}$.
%\end{example}

The homogeneity  property
$v_C(tx) = t^{-n} v_C(x)$ for $t>0$, $x\in \inte(C^*)$ 
is immediate to check. Moreover, the transformation formula 
implies the following invariance property: 
if $g\in\GL(n,\R)$ and $K:=g(C)$, then 
$g^T(K^*)= C^*$ and 
\begin{equation}\label{eq:invareq:-KV}
 v_K(z) = |\det g| \, v_C(g^T z) \quad \mbox{ for $z\in\inte(K^*)$} .
\end{equation}

\begin{remark}\label{pro:char-fcts}
The function $\log v_C$ is strictly convex and essentially equals Nesterov and Nemirowski's 
universal self-concordant barrier function~\cite[\S2.5]{nesterov-nemirovski:94}, 
see \cite{guler:96} for the proof.
\end{remark}

The following is well known, e.g., see~\cite[Thm.~4.1]{guler:96}.
Appendix~\ref{se:A1} contains the proof for the sake of completeness. 

\begin{lemma}\label{le:partial_volume}
We have
$v_C(x) = n!\, \vol\big\{ y\in C \mid \langle x,y\rangle \le 1 \big\}$
for $x\in \inte(C^*)$. 
\end{lemma}

The following %technical result 
is essential for the proof of 
Theorem~\ref{th:main-mixed-real}. 

\begin{prop}\label{pro:crucial-trick}
Let $C\subseteq\R^n$ be a proper cone. 
Then we have for $b_1,\ldots,b_n \in C^*$. 
$$
|\det[b_1,\ldots,b_n]| %\|b_1\wedge\ldots\wedge b_n\| 
\cdot v_C (b_1+\ldots +b_n) \le 1 .
$$
This bound is optimal. 
\end{prop}

\begin{proof}
We denote by $\cone(b_1,\ldots,b_n)\subseteq C^*$ the convex cone generated 
by $b_1,\ldots,b_n$. Without loss of generality, 
we may assume that $b_1,\ldots,b_n \in C^*$ is a basis of $\R^n$.
Let $b^*_1,\ldots,b^*_n$ denote its dual basis, that is 
$\langle b^*_i, b_j\rangle = \delta_{ij}$. In matrix terminology, this means 
$[b^*_1,\ldots,b^*_n]^T [b_1,\ldots,b_n] = I_n$, hence 
\begin{equation}\label{eq:vol-dual}
 \det [b^*_1,\ldots,b^*_n] \det[b_1,\ldots,b_n] = \pm 1 .
\end{equation}
The definition of the dual basis implies that 
$\cone(b^*_1,\ldots,b^*_n)$ is the dual cone of  $\cone(b_1,\ldots,b_n)$. 
Therefore, by duality, we get  
$$
  C \subseteq \cone(b_1,\ldots,b_n)^* = \cone(b^*_1,\ldots,b^*_n) .
$$
Put $d := b_1 + \ldots +b_n$ and 
let $y\in C$ such that $\langle d,y \rangle \le 1$. 
Since 
$C \subseteq \cone(b^*_1,\ldots,b^*_n)$,  
we can write 
$y=\sum_i t_i b^*_i$ 
with $t_i\ge 0$. Moreover
$\sum_i t_i = \langle d ,y\rangle \le 1$.
Thus we have shown the inclusion
$$
 K :=\{ y \in C \mid \langle d ,y \rangle \le 1\} \subseteq \conv\{0,b^*_1,\ldots,b^*_n\} . 
$$
This implies the inequality of volumes
\begin{equation*}
% \vol_n\{ y \in C \mid \langle d ,y \rangle \le 1\} &\le\ \vol_n\conv\{0,b^*_1,\ldots,b^*_n\} \\
\vol_n K \le\ \vol_n\conv\{0,b^*_1,\ldots,b^*_n\} 
  = \frac{1}{n!} |\det[b^*_1,\ldots,b^*_n]| .
\end{equation*}
Multiplying with $n!\, |\det[b_1,\ldots,b_n]|$, 
using \eqref{eq:vol-dual} 
and taking into account Lemma~\ref{le:partial_volume},
the assertion follows.

The optimality is attained for $C=\R^n_{>0}$ and $b_i=d_i e_i$ with $d_i >0$. 
Indeed, we have 
\begin{equation*}
 |\det [b_1,\ldots,b_n] |\cdot v_C(d) = d_1\cdot\ldots\cdot d_n\ (d_1\cdot\ldots\cdot d_n)^{-1} =  1 . 
 \qedhere 
\end{equation*}
%where we used Example~\ref{ex:v-pos-orthant}.
\end{proof}

%The characteristic function satisfies the  following remarkable properties, 
%see \cite[I.3]{faraut_koranyi:94} and \cite{guler:96}. 
%{\tt Not needed for the purposes of this paper.}
%%Let $C\subseteq\R^n$ be a proper cone. 
%\begin{enumerate}
%\item $v_C(x)$ goes to infinity when $x$ approaches the boundary of $C^*$.
%\item $v_C$ is analytic and strictly logarithmically convex. 
%\item $\log v_C$  is a self-concordant barrier function for $C^*$, 
%\end{enumerate}

\subsection{Vertices and normal fan of sums of polytopes}\label{se:NF}
%Normal cones at vertices of sums of polytope}

We recall here some basic facts about polytopes and their normal fans; 
see~\cite[\S7.1]{ziegler:95} for more details. 

%A convex cone $C\subseteq\R^n$ is called {\em proper} if it is $n$-dimensional 
%and pointed, i.e., contained in a halfspace. 
Let $P\subseteq\R^n$ be a full-dimensional polytope and $v$ be a vertex of~$P$. 
The {\em cone $P_v$ of~$P$ at $v$} is defined as the convex cone generated by $P-v$. 
It is a proper cone. The dual cone of $P_v$, 
also called the {\em inner normal cone} of $P$ at $v$, is defined as 
$$
 P_v^* := \{ y\in\R^n \mid \forall x \in P\ \langle x-v,y\rangle \ge 0 \} .
$$
The cone $P_v^*$ is also  proper.
The union over all $P_v^*$ equals $\R^n$. Moreover, for $v_1\ne v_2$, we have 
$\dim(P_{v_1}^* \cap P_{v_2}^*) < n$. 
In fact, the $P_v^*$ are the $n$-dimensional cones of the normal fan of~$P$.

We will need the following result. 

\begin{lemma}\label{le:V-sum-polytope}
Let $P_1,\ldots,P_n$ be polytopes in $\R^n$. 
There is an injective map 
$$
 \Vt(P_1+\ldots+P_n) \to \Vt(P_1)\times\ldots \Vt(P_n),\, v\mapsto (v_1,\ldots,v_n)
$$
satisfying $v=v_1+\ldots+v_n$. 
Moreover, if we denote by $\Pi_i$ the cone of 
$P_i$ at the vertex $v_i$, then 
$\Pi := \Pi_1+\ldots+\Pi_n$ is the cone of $P_1+\ldots+P_n$ at the vertex $v_1+\ldots+v_n$. 
In particular, 
$\Pi^*= \Pi_1^* \cap\ldots\cap \Pi_n^*$.
\end{lemma}

\begin{proof}
%For the following, see \cite[\S1.7]{schneider:14}. 
To a nonzero weight $\omega\in\R^n$ we assign the face of $P_i$, given by 
$$
 F(P_i,\omega) := \big\{ w\in\R^n \mid \langle w,\omega\rangle = \min_{w'\in P_i}  \langle w',\omega\rangle \big\} .
$$
We have by \cite[Thm.~1.7.5]{schneider:14} 
$$
 F(P_1+\ldots+P_n,\omega) = F(P_1,\omega) + \ldots + F(P_n,\omega) .
$$
Suppose that $F(P_1+\ldots+P_n,\omega) =\{v\}$ is a vertex. 
Then all $F(P_i,\omega)=\{v_i\}$ are vertices and $v=v_1+\dots +v_n$.
The $v_i$ are uniquely determined by~$v$, see~\cite[Prop.~2.1]{fukuda:04}.
Then the map $v\mapsto (v_1,\ldots,v_n)$ is as required.
The remaining assertions are clear.
\end{proof}

Lemma~\ref{le:V-sum-polytope} implies %that 
$V_0(P_1+\ldots+P_n) \le V_0(P_1)\cdots V_0(P_n)$.
This bound is sharp, see~\cite{fukuda-weibel:07,karavelas-tzanaki:11}. 
%{\tt Include refs \cite{fukuda-weibel:07,karavelas-tzanaki:11} }

\section{Random intersections in products of projective spaces}

\subsection{The kinematic formula}

We specialize here the general kinematic formula for homogeneous spaces 
from~\cite[Thm.~A.2]{BL:19} to the case of products of real projective spaces (Theorem~\ref{th:KF}). 
For this purpose, we define the average scaling factor 
and we explain how to bound it in Lemma~\ref{le:select}. 

Consider the product 
$\Omega:=\proj^{m_1}\times\cdots\times\proj^{m_n}$ of real projective spaces.
The product $G:=O(m_1+1)\times\cdots\times O(m_n+1)$ of orthogonal groups acts transitively on $\Omega$. 
So $\Omega$ is a homogeneous space and we have an induced transitive action of $G$ on the tangent bundle of $\Omega$. 
We focus on the special hypersurfaces $H_1,\ldots,H_n$ of $\Omega$ of the following shape
\begin{equation}\label{eq:shape}
 H_1 := \proj^{m_1-1}\times \proj^{m_2}\times\cdots\times\proj^{m_n},\ldots, 
 H_n := \proj^{m_1}\times\proj^{m_2}\cdots\times\proj^{m_n-1} .
\end{equation}
They are determined upon selecting hyperplanes $\proj^{m_i-1}$ in each $\proj^{m_i}$. 
%Let $Z\subseteq \Omega$ be an $n$-dimensional smooth submanifold. 
Our goal is to investigate the average cardinality of the intersection 
$Z\cap H_1\cap\ldots\cap H_n$ of  an $n$-dimensional smooth submanifold $Z\subseteq \Omega$ 
with random $H_i$, which are defined by replacing the fixed $\proj^{m_i-1}$ by 
independently chosen uniform random hyperplanes in~$\proj^{m_i}$. 

%Write $e_0:=[1:0:\colon:0]$
Fix a distinguished point~$\omega\in\Omega$ and denote by $K$ the stabilizer group of~$\omega$. 
E.g., take $\omega_i=[1:0\ldots:0]$ for all~$i$.
Notice that we have an induced action of $K$ on the tangent space $T :=T_{\omega} \Omega$, which 
we can identify with the standard action of $K=O(m_1)\times\cdots\times O(m_n)$ on 
$T=\R^{m_1}\times\cdots\times\R^{m_n}$.  This induces an action of $K$ on the 
Grassmann manifold $\Gr(d,T)$ of linear subspaces of  $T$ with codimension~$d$. 
Note that this action is transitive if $n=1$, but not for $n\ge 2$. %, this is not case. 

%Following \cite[A.5.1]{BL:19}, we assign to $Z$ a function 
%$\s_Z\colon Z\to [0,1]$ measuring the average relative position of $Z$ with respect to 
%uniformly randomly $H_1,\ldots,H_n$. For this purpose, 

We assign to an $n$-dimensional smooth submanifold $Z\subseteq \Omega$ a map
\begin{equation}\label{eq:defmap}
 Z \to \Gr(n,T)/K,\ p \mapsto KgN_pZ 
\end{equation}
as follows. 
%{\tt First recall $\sigma$; see \S\ref{se:ASF}.}
For given $p\in Z$ choose any 
$g\in G$ such that $gp =\omega$. 
The induced action of $g$ maps the tangent space 
$T_{p} \Omega$ to $T_{\omega} \Omega=T$. 
This transports the normal subspace
$N_pZ\subseteq T_p\Omega$ of~$Z$ at~$p$ to  
$gN_pZ\subseteq T$. 
Note that the $K$-orbit of the subspace 
$gN_pZ$ does not depend on the choice of~$g$, 
which shows that the map~\eqref{eq:defmap}
is well defined. 

We call the submanifold $Z$ {\em cohomogeneous} if 
the map~\eqref{eq:defmap} is constant; 
see \cite[A.5.1]{BL:19} and \cite{matis-thesis:22}. 
For instance, a product 
$Z=\cL_1\times\ldots\times\cL_n$ 
of lines $\cL_i$ in $\proj^{m_i}$ 
is cohomogeneous: indeed, 
the map~\eqref{eq:defmap} sends any point $p\in Z$ to the $K$-orbit of 
$\R\times \ldots \times \R$.
%the hypersurface~$Z=H_1$ corresponds to $H$-orbit 
%consisting of the subspaces of the form %$L_1\times 0\times \cdots\times 0\subseteq T$, 

\begin{defi}\label{def:asf}
The \emph{average scaling factor function} of the $n$-dimensional 
submanifold $Z$ of $\proj^{m_1}\times\cdots\times\proj^{m_n}$ 
is the function $\overline{\sigma}_Z\colon Z\to [0,1]$ 
defined at $p\in Z$ by 
$$ 
 \overline{\s}_Z(p) := \E_{L_i}\s(gN_pZ,L_1\times \ldots \times L_n) ,
$$
where $g\in G$ satisfies $gp =\omega$, and  
the expectation is taken over uniformly random lines $L_i$ in 
$T=\R^{m_1}\times\cdots\times\R^{m_n}$; 
see~\eqref{eq:def-sigma} for the definition of~$\sigma$. 
\end{defi}

Note that due to the averaging over the $K$-orbit, the choice of $g$ is irrelevant. 
The above definition is consistent with the one in \cite[Def.~A.1]{BL:19}, since
\begin{eqnarray}\notag%\label{eq:s-equal}
\lefteqn{\s(gN_pZ,L_1\times \ldots \times L_n) }  \\ \label{eq:s-equal}
 &= \s(gN_pZ,L_1\times 0\times \cdots\times 0,\ldots,0\times \cdots\times 0\times L_n) 
\end{eqnarray}
by \eqref{eq:reduce-sigma}; indeed note that the $n$ lines 
$L_1\times 0\times \cdots\times 0$, ...
are pairwise orthogonal.
%,\ldots $0\times \cdots\times 0\times L_n$ 
%are pairwise orthogonal.  

%The quantity $\s$ on the right-hand side is defined in \cite[eq. (3.5)]{BL:19}; 
%see~\S\ref{se:ASF}. 

We introduce the notation 
$$ 
\rho_n := \E\|x\| = \sqrt{2}\, \frac{\Gamma(\frac{n+1}{2})}{\Gamma(\frac{n}{2})} \ \le\ \sqrt{n}
$$  
for standard Gaussian $x\in\R^n$ and 
note that \cite[Lemma~2.25]{Condition}, 
\begin{equation}\label{eq:projvol}
  \frac{\vol(\proj^{m_i-1})}{\vol(\proj^{m_i})} = \frac{1}{\sqrt{\pi}} \frac{\Gamma(\frac{m_i+1}{2})}{\Gamma(\frac{m_i}{2})} 
   = \frac{1}{\sqrt{2\pi}} \rho_{m_i} . 
\end{equation}

We can now explicitly state the kinematic formula for products of real projective spaces.
%The kinematic formula in~\cite[Thm.~A.2]{BL:19} has the following consequence. 

\begin{thm}\label{th:KF}
For any $n$-dimensional submanifold $Z$ of 
$\proj^{m_1}\times\cdots\times\proj^{m_n}$,
we have 
$$
 \E_{g\in G} \#(Z\cap g_1H_1\cap \ldots \cap g_nH_n) = 
   (2\pi)^{-\frac{n}{2}} \rho_{m_1}\cdots \rho_{m_n} \int_Z \overline{\s}_Z \, dZ ,
$$
where the hypersurfaces $H_i$ are defined in~\eqref{eq:shape}. 
\end{thm}

\begin{proof} 
If $\s_K\colon Z\times H_1\times\ldots \times H_n\to [0,1]$ denotes 
the average scaling function from~\cite[Def.~A.1]{BL:19}, then~\cite[Thm.~A.2]{BL:19} states that 
\begin{eqnarray*}
  \lefteqn{\E_{g\in G} \#(Z\cap g_1H_1\cap \ldots \cap g_nH_n) } \\
 &=  \frac{1}{\vol(\Omega)^n} \int_{Z\times H_1 \times \ldots \times H_n} \s_K \, d(Z\times H_1 \times \ldots \times H_n) .
\end{eqnarray*}
By $K$-invariance and~\eqref{eq:s-equal}, we have  
$\s_K(z,y_1,\ldots,y_n)= \overline{\s}_Z(z)$ 
for all $z\in Z$ and $y_i \in H_i$. 
Therefore, 
\[
  \E_{g\in G} \#(Z\cap g_1H_1\cap \ldots \cap g_nH_n) = 
  \frac{\vol(H_1) \cdots \vol(H_n)}{\vol(\Omega)^n}  \int_{Z} \overline{\s}_Z \, dZ .
\]
Finally, \eqref{eq:projvol} gives 
\begin{equation*}
 \frac{\vol(H_1) \cdots \vol(H_n)}{\vol(\Omega)^n} = \prod_{i=1}^n \frac{\vol(\proj^{m_i-1})}{\vol(\proj^{m_i})} 
 =  \frac{\rho_{m_1}\cdots \rho_{m_n}}{ (2\pi)^{\frac{n}{2}} } , 
\end{equation*}
which completes the proof.
\end{proof}

\begin{example}
A product $Z=\cL_1\times\ldots\times\cL_n$ 
of lines $\cL_i$ is cohomogeneous 
and we have $\overline{\s}_Z = (2/\pi)^{n/2} (\rho_{m_1}\cdots \rho_{m_n})^{-1}$ 
by Theorem~\ref{th:KF}.
\end{example}

We shall focus on submanifolds $Z$ arising as the image of an injective map 
\begin{equation}\label{eq:Gmap-psi}
 \psi\colon U \to \proj^{m_1}\times\cdots\times\proj^{m_n},\ 
 \psi(x) := (\psi_1(x),\ldots,\psi_n(x)) ,
\end{equation}
where the $\psi_i\colon U \to \proj^{m_i}$ are smooth maps 
defined on an open subset $U\subseteq \R^n$. 
Let us denote by 
$$
 J\psi(x) := \sqrt{\det ((D_x\psi)^T D_x\psi)}
$$ 
the {\em absolute Jacobian} of $\psi$ at $x$. 
The transformation formula implies that 
\begin{equation}\label{eq:topsy}
 \int_Z \overline{\s}_Z \, dZ = \int_U \overline{\s}_Z(\psi(x)) J\psi(x) \, dx.
\end{equation}
We next analyze the integrand on the right-hand side more closely. 
%In this situation, we can express the average scaaling factor function $\s_Z$ of $Z$ as follows. 

\begin{lemma}\label{le:asf}
Let $x\in U$ and put $T_i := T_{\psi_i(x)}\proj^{m_i}$. 
Let $\lambda_1,\ldots,\lambda_n$ be  independent standard Gaussian 
linear forms on $T_i$. This defines the random linear forms 
$\lambda_i\circ D_x \psi_i$ on $\R^n$. 
Then 
\begin{eqnarray*}
 \lefteqn{ \rho_{m_1}\cdots \rho_{m_n} \overline{\s}_Z(\psi(x)) J\psi(x) } \\ 
  &= \E_{\lambda_1,\ldots,\lambda_n} 
   \left\| (\lambda_1 \circ D_x \psi_1) \wedge\ldots\wedge (\lambda_n\circ D_x \psi_n) \right\| .
\end{eqnarray*}
\end{lemma}

\begin{proof}
To simplify notation, we assume w.l.o.g.\ that $\omega=\psi(x)$ is the distinguished point.  
We also identify $T_i$ with $\R^{m_i}$. 
For $u_i\in T_i$ with $\|u_i\|=1$  consider the line $L_i=\R u_i$
and the orthogonal projection $p_{i}\colon T_i \to L_i$, which is 
is given by $p_{i}(w) = \mu_i(w)u_i$ with the 
linear form on~$T_i$ defined by 
$\mu_i(w) := \langle w, u_i\rangle$.
Thus the orthogonal projection 
$p_L\colon T_1\times\cdots\times T_n\to L_1\times\cdots\times L_n$ 
is described by $\mu_1,\ldots,\mu_n$. This implies that 
%to the tangent space~$T_pZ$.
\begin{equation}\label{eq:mun}
 |\det(p_L \circ D_x \psi)| = 
 \left\| (\mu_1 \circ D_x \psi_1) \wedge\ldots\wedge (\mu_n\circ D_x \psi_n) \right\| .
\end{equation}
%where $p_L\colon T_pZ \to L_1\times\cdots\times L_n$ 
%denotes the restriction of the orthogonal projection 
%$T_1\times\cdots\times T_n\to L_1\times\cdots\times L_n$ 
%to the tangent space~$T_pZ$. 
%For lines $L_i$ in $T_i$,  
%we denote by $p_L\colon T_1\times\cdots\times T_n\to L_1\times\cdots\times L_n$ 
%the orthogonal projection 
%and by $p'_L\colon T_pZ \to L_1\times\cdots\times L_n$ its restriction 
%to the tangent space~$T_pZ$. 
On the other hand,
according to Proposition~\ref{prop:sigma-symm}, we have 
$$
% \bar{\s}_Z(p) = 
\s(L_1\times\cdots\times L_n,N_pZ) = |\det p'_L| ,
$$
where $p'_L\colon T_pZ \to L_1\times\cdots\times L_n$ 
denotes the restriction of $p_L$ to~$T_pZ$. 
%the orthogonal projection 
%$T_1\times\cdots\times T_n\to L_1\times\cdots\times L_n$ 
Applying the determinant to the composition of $D_x\psi$ with~$p'_L$, we get 
$$
  J\psi(x) \, |\det p'_L| =  |\det(p_L \circ D_x\psi)| . 
$$
By averaging over random lines $L_i$, we deduce from the definition of 
$\overline{\s}_Z$ and the above that 
\begin{eqnarray*}
 \lefteqn{ J\psi(x) \overline{\s}_Z(p) =  J\psi(x)\, \E_{L_i} \s(N_pZ,L_1\times\cdots\times L_n) }\\
  &= J\psi(x)\, \E_{L_i} |\det p'_L| = \E_{L_i}|\det(p_L \circ D_\psi)| .
\end{eqnarray*} 
%Let us write $L_i=\R u_i$ with $\|u_i\|=1$ and define 
%the linear form $\mu_i$ on $T_i$ by 
%$\mu_i(w) := \langle w, u_i\rangle$. 
%Then we have 
%$$
% |\det(p_L \circ D_x \psi)| = 
% \left\| (\mu_1 \circ D_x \psi_1) \wedge\ldots\wedge (\mu_n\circ D_x \psi_n) \right\| .
%$$ 
Finally, a standard Gaussian linear form on $T_i$ is obtained as $\lambda_i = r_i \mu_i$
with independent random variables $r_i$ and $u_i$, where $u_i$ 
is uniformly random in the unit sphere of $T_i$ 
and $r_i^2$ is $\chi^2$-distributed with $m_i$ degrees of 
freedom. Thus $\E r_i = \rho_{m_i}$.
Altogether, we obtain, using~\eqref{eq:mun}, 
\begin{eqnarray*}
 \lefteqn{\rho_{m_1}\cdots \rho_{m_n} J\psi(x) \overline{\s}_Z(p) =\rho_{m_1}\cdots \rho_{m_n} \E |\det(p_L \circ D_\psi)| } \\
 &=\rho_{m_1}\cdots \rho_{m_n} \E \left\| (\mu_1 \circ D_x \psi_1) \wedge\ldots\wedge (\mu_n\circ D_x \psi_n) \right\| \\
 &=\E \left\| (\lambda_1 \circ D_x \psi_1) \wedge\ldots\wedge (\lambda_n\circ D_x \psi_n) \right\|  ,
\end{eqnarray*} 
which completes the proof. 
\end{proof}

\subsection{Bounding the average scaling factor}\label{se:Bd}

In order to bound the quantity in Lemma~\ref{le:asf},
we use affine charts for the product of projective spaces.  
Let $y_{i0},\ldots,y_{im_i}$ be coordinates for~$\proj^{m_i}$.
Fix $0 \le r_i\le m_i$ for $i=1,\ldots,n$, and consider
the inverse of the affine chart 
$\pi_{i r_i}\colon \R^{m_i} \to (\proj^{m_i})_{y_{ir_i}\ne 0}$, 
see Subsection~\ref{se:Pproj}. 
We describe the maps $\psi_i $ from \eqref{eq:Gmap-psi}
in these charts by smooth functions 
defined on open subsets of $\R^n$, 
\begin{equation}\label{eq:def-varphi}
\varphi_{ir_i} \colon \R^n\supseteq U_{ir_i}  \to \R^{m_i} ,
\end{equation}
satisfying  
$\psi_i := \pi_{ir_i} \circ \varphi_{ir_i}$.
In order to simplify notation, we assume w.l.og.\  
$r_i=0$ and write $\pi_{i} := \pi_{i 0}$, $\varphi_{i}:= \varphi_{i0}$.  
In these charts, the combined map $\psi$ of \eqref{eq:Gmap-psi} is represented by 
a map 
$$
 \varphi\colon U\to \R^{m_1}\times\cdots\times\R^{m_n}, \, 
 \varphi(x) = (\varphi_{1}(x),\ldots,\varphi_{n}(x)) 
$$
defined on some open subset $U\subseteq\R^n$. 
We view the derivative  $M(x):=D_x\varphi$ as a matrix of format $(m_1+\ldots+m_n)\times n$ 
with blocks $M_i(x) := D_x\varphi_i\in\R^{m_i\times n}$. 
For $1\le j_i\le m_i$, $i=1,\ldots,n$, we denote by $M(x)_{j_1,\ldots,j_n}$ 
the $n\times n$ submatrix of $M(x)$ obtained by selecting in the $i$th block the $j_i$th row.

\begin{lemma}\label{le:select}
Let $x\in U$ such that $[y_i]:=\psi_i(x) \in   (\proj^{m_i})_{y_{i0} \ne 0}$ for all~$i$. 
Then 
$$
  \rho_{m_1}\cdots \rho_{m_n} \,\overline{\s}_Z (\psi(x))\, J\psi(x) \ \le\ 
 \sum_{j_1,\ldots,j_n} |\det M(x)_{j_1,\ldots,j_n}| ,
$$
where the sum is over $n$-tuples 
$(j_1,\ldots,j_n) \in [m_1]\times\ldots\times [m_n]$.
%, where $0\le j_i\le m_i$ and $j_i \ne r_i$. 
\end{lemma}

\begin{proof}
From $\psi_i = \pi_i \circ \varphi_i$ we get 
$D\psi_i = D\pi_i \circ D\varphi_i$, 
where we drop arguments for notational simplicity. 
Let $\lambda_i \colon T_i\to\R$ be a linear form on $T_i = T_{\psi_i(x)}\proj^{m_i}$.
Then, defining $w_i := \lambda_i \circ D \pi_i$,
$$
 \lambda_i \circ D\psi_i = \lambda_i \circ D\pi_i \circ D\varphi_i = w_i \circ D\varphi_i .
$$ 
If we identify  
$\lambda_i \circ D_x\psi_i$
with a vector in $\R^{n}$ 
and $w_i$ with a vector in $\R^{m_i}$, then 
we have the matrix product 
of formats $n\times \sum_i {m_i}$ and $ \sum_i {m_i}\times n$, 
\begin{equation}\label{eq:R(x)}
R(x):= 
\begin{bmatrix}
(\lambda_1 \circ D_x \psi_1)^T \\
\vdots\\
(\lambda_n \circ D_x \psi_1)^T
\end{bmatrix}
%R(x):=
=
\begin{bmatrix}
w_1^T     &         0 & \ldots 0 \\
0        &      w_2^T & \ldots 0 \\
\vdots &  \vdots & \vdots \\
0        &          0 & w_n^T
\end{bmatrix}
\cdot
\begin{bmatrix}
M_1(x)\\
\vdots\\
M_n(x)
\end{bmatrix} .
\end{equation}
Lemma~\ref{le:asf} tells us that 
$$
  \rho_{m_1}\cdots \rho_{m_n} \overline{\s}_Z(\psi(x)) J\psi(x) =\E_{\lambda_i} |\det R(x))| ,
$$
where the expectation is over independent standard Gaussian~$\lambda_i$. 
Note that the resulting random vector  
$w_i := \lambda_i \circ D\pi_i$
is not standard Gaussian anymore.
However $\|D\pi_i\|\le 1$ by \eqref{eq:der-pi},  
and Lemma~\ref{le:SN} below imply that 
$\E w_{ij}^2 \le 1$
for the $j$th component $w_{ij}$ of $w_i$.

From Cauchy-Binet, we obtain from \eqref{eq:R(x)}
$$
 (\det R(x))^2 = \sum_{j_1,\ldots,j_n}  w_{1 j_1}^2 \cdots w_{n j_n}^2 
   (\det M(x)_{j_1,\ldots,j_n} )^2, 
$$ 
where the sum is over all $(j_1,\ldots,j_n) \in [m_1]\times\ldots\times [m_n]$.
Taking expectations yields
$$
 \E_w (\det R(x))^2 \ \le\  \sum_{j_1,\ldots,j_n}  (\det M(x)_{j_1,\ldots,j_n} )^2 .
$$
We conclude that 
\begin{eqnarray*}
 \E_w |\det R(x))| \ \le&\ \big(\E_w (\det R(x))^2 \big)^{\frac12} \\
   \le&\ \sum_{j_1,\ldots,j_n}  |\det M(x)_{j_1,\ldots,j_n}| ,
\end{eqnarray*}
which completes the proof.
\end{proof}

\begin{lemma}\label{le:SN}
Let $A\in\R^{p\times m}$ with $\|A\|\le 1$. 
If $y\in\R^p$ is standard Gaussian, then 
the random variable $z:=yA$ satisfies 
$\E |z_j|^2 \le 1$ for all $j$. 
\end{lemma}

\begin{proof}
From $z_j = \sum_i y_i a_{ij}$ we get 
$z_j^2 = \sum_{i,k} y_i y_k a_{ij} a_{kj}$.
Hence 
$\E z_j^2 = \sum_{i} a_{ij}^2$.
Finally, $ \sum_{i} a_{ij}^2= \| A(e_j)\|^2 \le \| A\|^2 \le 1$. 
\end{proof}

\section{Mixed random fewnomial systems}\label{se:mixed}

We provide here the proofs of the assertions in the introduction. 
Let us first introduce some notation. 

We assign to a real valued function $c\colon A\to\R$ on 
a finite nonempty subset $A\subseteq\R^n$ 
%and a coefficient function $c\colon A\to\R$ the 
the real analytic function 
$F_{A,c}\colon\R^n\to\R$
\begin{equation}\label{eq:def_F}
 F_{A,c}(w) := \sum_{a\in A} c(a) e^{\langle a,w\rangle} .
\end{equation}
In the special case where $A$ consists of integer vectors, 
$F_{A,c}$ arises from the Laurent polynomial
$f_{A,c}(x) = \sum_{a\in A} c(a) x^a$ 
%\end{equation*}
by a substitution: $F_{A,c}(w) =f_{A,c}(e^w)$.
Generally, we have the following equivariance property: for $g\in\GL(n,\R)$ and $b\in\R^n$, 
\begin{equation}\label{eq:invar-F}
 F_{A+b,b.c}(w) = e^{\langle b,w\rangle} F_{A,c}(w) ,\ 
 F_{g(A),g.c}(w) = F_{A,c}(g^Tw) , 
\end{equation}
where $b.c(a) := c(a-b)$ and $(g.c)(a) := c(g^{-1}a)$. 

Suppose now we have $n$ such analytic functions encoded by 
$c_i\colon A_i\to\R$, for $i=1,\ldots,n$. 
Throughout, we denote by~$t_i$ the cardinality of $A_i$ and by 
$P_i$ its convex hull. 
We are interested in the number $N$ of nondegenerate zeros $w\in\R^n$
of the system 
\begin{equation}\label{eq:def_F_Syst}
 F_{A_1,c_1}(w)=0,\ldots, F_{A_n,c_n}(w)=0.
\end{equation}
Our goal is to study the expected number of nondegenerate zeros 
for random coefficient functions. More specifically,
we denote by $\E(A_1,\ldots,A_n)$ the expectation of $N$, 
when all the coefficients~$c_i(a)$, for $i\in [n]$ and $a_i\in A_i$,  
are independent standard Gaussians. 
Clearly, $\E(A_1,\ldots,A_n)$ is invariant under permutations of the $A_i$. 
Also, $\E(A_1,\ldots,A_n)=0$ if $t_i=1$ for some~$i$. 
Moreover, we have 
$\E(A_1,\ldots,A_n)=0$ if $\dim(P_1+\ldots +P_n)<n$, 
see Lemma~\ref{le:psi-inj}.
 
Equation~\eqref{eq:invar-F} implies the following invariance properties
\begin{equation}\label{eq:invar-N}
  \begin{split}
 \E(A_1+b_1,\ldots,A_n+b_n) &= \E(A_1,\ldots,A_n), \\
 \E(g(A_1),\ldots,g(A_n)) &= \E(A_1,\ldots,A_n) ,
 \end{split}
\end{equation}
where  $b_1,\ldots,b_n\in\R^n$ and $g\in\GL(n,\R)$. 
In particular, $\E$ is invariant under replacing $A_i$ by $\lambda_i A_i$ 
for $\lambda_i\in\R^\times$.

Our main result is Theorem~\ref{th:main-mixed-real} stated in the introduction. 
Note that it gives the correct answer $\E(A_1,\ldots,A_n)=0$ if $t_i=1$ for some~$i$.

\begin{example}
In the case $t_1=\ldots=t_n=2$, the $P_i$ are segments. If they are linearly independent, 
$P_1+\ldots+P_n$ is a parallelepiped with $2^n$ vertices. Thus, 
Theorem~\ref{th:main-mixed-real} gives 
$\E(A_1,\ldots,A_n) \le (2/\pi)^{\frac{n}{2}}$. 
This can be easily verified directly as follows. 
%\begin{example}\label{ex:t=2}
Suppose 
$A_i=\{a_i,b_i\}$,
where $b_1-a_1,\ldots,b_n-a_n$ are linearly independent.
We claim that $\E(A_1,\ldots,A_n)=2^{-n}$. %if $b_1-a_1,\ldots,b_n-a_n$ are linearly independent. 
For showing this, by the invariance properties~\eqref{eq:invar-N}, it suffices to consider the 
case where $A_i=\{0,e_i\}$. Then~\eqref{eq:def_F_Syst} 
amounts to the system 
$c_i(0) + c_i(e_i) e^{w_i} = 0$, for $i=1,\ldots,n$, 
which has a solution iff  $c_i(0)c_i(e_i) <0$, for all~$i$. 
This happens with probability $2^{-n}$, hence indeed
$\E(A_1,\ldots,A_n)=2^{-n}$. 
\end{example}

%This was shown to be true (and optimal) in the unmixed case for $n=1$
%in~\cite{jindal_et_al:20}, by removing a $\log t$ factor in a bound in \cite{BETC:19}. 

\subsection{Proof of Theorem~\ref{th:main-mixed-real}}\label{se:pf-main-mixed}

Let us look at a special instance of~\eqref{eq:Gmap-psi}.
To the given finite nonempty subsets $A_1,\ldots,A_n\subseteq\R^n$, 
we assign the maps
$$
 \psi_i\colon \R^n_{>0} \to \proj(\R^{A_i}) \simeq \proj^{m_i}, \ 
  \psi_i(x) := [x^{a_i}]_{a_i\in A_i} ,
$$
where $m_i:=\#A_i -1$. 
Recall that $P_i$ denotes the convex hull of $A_i$ and  put $P:=P_1+\ldots+P_n$. 
We consider the combined map
\begin{equation}%\label{eq:Gmap-psi}
 \psi\colon \R^n_{>0} \to \proj^{m_1}\times\cdots\times\proj^{m_n},\ 
 \psi(x) := (\psi_1(x),\ldots,\psi_n(x)) .
\end{equation}

\begin{lemma}\label{le:psi-inj}
The map $\psi$ is injective iff $P$ is $n$-dimensional.
Moreover, if $P$ is not $n$-dimensional, then 
$\mathrm{rank} D_x\psi<n$ for all $x\in \R^n_{>0}$. 
\end{lemma}

\begin{proof}
Assume $\psi(\exp(w)) = \psi(\exp(w'))$ for $w\ne w'\in\R^n$  
Then there are $c_i\in\R$ such that for all $a_i \in A_i$ we have that  
$\langle a_i, w-w' \rangle = c_i$. Hence, 
$\langle x, w-w' \rangle = c_i$ for all $x_i\in P_i$. 
It follows that $\langle x, w-w' \rangle = c_1+\ldots+c_n$ for all 
$x\in P$. Hence $\dim P <n$. 

Conversely, assume there is a nonzero $w\in\R^n$ and $c\in\R$ such that 
$\langle x, w\rangle =c$ for all $x\in P$. 
Then there are $c_i\in\R$ such that 
$\langle x_i, w\rangle =c_i$ for all $x_i\in P_i$. 
It follows that for any $x\in\R^n_{>0}$ and any $s\in\R$ we have 
\begin{eqnarray*}
 \psi_i(e^{sw}x) &= [ (e^{sw})^{a_i} x^{a_i} ]_{a_i \in A_i} = [ e^{s\langle a_i, w\rangle} x^{a_i}]_{a_\in A_i} \\
   &=[ e^{sc_i}x^{a_i} ]_{a_\in A_i}= \psi_i(x).
\end{eqnarray*}
Hence $\psi$ is not injective. 
Moreover, $w$ is in the kernel of the 
derivative of $\psi_i$ at $x$. 
\end{proof}

We denote by $Z$ the image of $\psi$. 
Then we can write 
$$
 \E(A_1,\ldots,A_n) = \E_{g\in G} \#(Z\cap g_1H_1\cap \ldots \cap g_nH_n) ,
$$ 
where the hypersurfaces $H_i$ are defined in~\eqref{eq:shape}. 
By Theorem~\ref{th:KF} and \eqref{eq:topsy},  
this can be expressed as 
\begin{equation}\label{eq:NAI}
 \E(A_1,\ldots,A_n) = (2\pi)^{-\frac{n}{2}} \rho_{m_1}\cdots \rho_{m_n} \int_{\R^n_>}(\overline{\s}_Z\circ \psi) J\psi \, dx .
\end{equation}
We make the coordinate change $\R^n\to\R^n_{>0}, (w_1,\ldots,w_n)\mapsto x = (e^{-w_1},\ldots, e^{-w_n})$,
which has the absolute Jacobian $x_1\cdots x_n$, and obtain 
(slightly abusing notation) 
\begin{equation}\label{eq:Se-Int}
 \int_{\R^n_>}(\overline{\s}_Z\circ \psi) J\psi \, dx = \int_{\R^n} x_1\cdots x_n (\overline{\s}_Z\circ \psi) J\psi  \, dw . 
%\overline{\s}_Z(\psi \circ J\psi)(x) \, dw . 
\end{equation}

%Let $\cF$ denote the normal fan of $P:=P_1+\ldots+P_n$; see \cite[\S7.1]{ziegler:95}. 
%for more information on the normal cones of a polytope)
%The $n$-dimensional cones of $\cF$ correspond 
%bijectively to the vertices of $P$.
%The fan $\cF$ consists of pointed cones since we assume that $P$ %$P_1+\ldots+P_n$ 
%is $n$-dimensional. 

Recall from Subsection~\ref{se:NF} that each vertex $v$ of $P$ defines 
the inner normal cone $C_v:=P^*_v$. 
%that $n$-dimensional inner normal cones of $\cF$ 
%correspond bijectively to the vertices of $P$. 
We can write 
\begin{equation}\label{eq:decomp}
 \R^n = \bigcup_v C_v
\end{equation} 
as the union over the vertices~$v$ of $P$.
Moreover, we know that $\dim (C_v\cap C_{v'}) < n$ for different vertices $v,v'$.
Therefore, we can rewrite~\eqref{eq:Se-Int} as the sum
$$
 \sum_v \int_{C_v}  x_1\cdots x_n(\overline{\s}_Z\circ \psi) J\psi  \, dw . 
$$ 
over the $V_0$ many vertices $v$ of $P$.

Fix now a vertex~$v$ of $P$. 
%and write $\Pi :=P_v$ for the cone of $P$ at $v$. Thus $C_v=\Pi^*$. 
According to Lemma~\ref{le:V-sum-polytope}, 
there are vertices $v_i$ of~$P_i$, for $i=1,\ldots,n$, satisfying 
$v=v_1+\ldots+v_n$. Note that $a_i \in A_i$. 
%Moreover, if $\Pi_i$ denotes the cone of~$P_i$ at the vertex $v_i$, then 
%$\Pi = \Pi_1 +\ldots+ \Pi_n$.  

We define the map $\varphi_i\colon \R^n_{>0} \to \R^{A_i \setminus\{v_i\}}$ by 
$$
 \varphi_i(x) = ( x^{a_i- v_i})_{a_i \in A_i \setminus\{v_i\}} \in \R^{A_i\setminus\{v_i\}} \simeq \R^{m_i} .
$$
Note that $\varphi_i$ expresses $\psi_i$ in the affine chart %given by 
$$
\proj(\R^{A_i})_{y_{iv_i}\ne 0}\to \R^{a_i\in A_i \setminus\{v_i\}}, \, 
 [y_{ia_i}]_{a_i\in A_i} \mapsto \frac{1}{y_{iv_i}} (y_{i a_i})_{a_i\in A_i\setminus\{v_i\}} . 
% y_{iv_i}^{-1} (y_{i a_i})_{a_i\in A_i\setminus\{v_i\}} . 
$$
%which maps 
%$[y_{ia_i}]_{a_i\in A_i}$ to $y_{iv_i}^{-1} (y_{i a_i})_{a_i\in A_i \setminus\{v_i\}}$. 
So we are in the setting of Subsection~\ref{se:Bd} and 
$\varphi_i$ is an instance of~\eqref{eq:def-varphi}.
The rows of the matrix  $M(x) := D_x\varphi$ are labeled 
by the disjoint union $A_1\sqcup\ldots \sqcup A_n$ and $M(x)$ 
has $n$ columns. 
For any $n$-tuple $(a_1,\ldots,a_n)$ with  
$a_i \in  A_i \setminus\{v_i\}$, we denote by 
$M(x)_{a_1,\ldots,a_n}$ the $n\times n$ submatrix of $M(x)$, 
obtained by selecting from $M(x)$ the rows numbered by 
$a_1,\ldots,a_n$. 
We apply Lemma~\ref{le:select} to bound 
\begin{eqnarray*}
  \lefteqn{ \rho_{m_1}\cdots \rho_{m_n} \, \int_{C_v}  x_1\cdots x_n (\overline{\s}_Z\circ \psi) J\psi  \, dw} \\
  &\le \sum_{a_1,\ldots,a_n} 
  \int_{C_v}  x_1\cdots x_n |\det M(x)_{a_1,\ldots,a_n}| \, dw ,
\end{eqnarray*}
where the sum runs over all tuples $(a_1,\ldots,a_n)$ with  
$a_i \in  A_i \setminus\{v_i\}$. So there are 
$m_1\cdots m_n$ many summands.  
To prove Theorem~\ref{th:main-mixed-real}, 
it is  sufficient to show that 
\begin{equation}\label{eq:STS}
  \int_{C_v}  x_1\cdots x_n |\det M(x)_{a_1,\ldots,a_n}| \, dw \ \le\ 1 
\end{equation}
for each vertex $v$ and each selection $(a_1,\ldots,a_n)$. 

The component (row) of the derivative $D_x \varphi_i$ corresponding to 
$a_i \in  A_i \setminus\{v_i\}$ is given by 
$$
 (D_x \varphi_i)_{a_i} =  x^{a_i -v_i} (a_i -v_i) \diag(x_1^{-1},\ldots,x_n^{-1}) .
$$
Hence the  $n\times n$-submatrix 
$M(x)_{a_1,\ldots,a_n}$ of $M(x)$ 
%corresponding to the row selections $a_1,\ldots,a_n$,
is given by 
\begin{eqnarray*}
 \lefteqn{ M(x)_{a_1,\ldots,a_n} = } \\ 
 & \diag(x^{a_1 -v_1},\ldots, x^{a_n -v_n}) 
  \begin{bmatrix} a_1 -v_1 \\ \vdots \\ a_n -v_n \end{bmatrix}
  \diag(x_1^{-1},\ldots,x_n^{-1}) .
\end{eqnarray*}
Therefore, setting $b_i:=a_i -v_i$, we get 
$$
x_1\cdots x_n \det (M(x)_{a_1,\ldots,a_n}) = x^{b_1+\ldots + b_n} \det [b_1,\ldots,b_n] .
$$
Let us write $\Pi_i$ for the cone of $P_i$ at the vertex $v_i$. 
By definition, $b_i \in \Pi_i^*$. 
By Lemma~\ref{le:V-sum-polytope}, 
$\Pi:=\Pi_1+\ldots+\Pi_n$ equals the cone of the polytope 
$P=P_1+\ldots+ P_n$ at the vertex $v=v_1+\ldots+v_n$. 
Hence $b_i \in \Pi_i^* \subseteq \Pi_1^* \cap\ldots\cap \Pi_n^* = \Pi^* =C_v$.
%Thus 
%$C_v=\Pi^*=\Pi_1^* \cap\ldots\cap \Pi_n^*$
%and we see that $b_1,\ldots,b_n \in C_v$.
%$\Pi^*=\Pi_1^* \cap\ldots\cap \Pi_n^*$.
%since $\Pi_i \subseteq \Pi$. 

We can therefore rewrite the left-hand side of \eqref{eq:STS} as 
\begin{equation}\label{eq:anschluss}
  \begin{split}
    \int_{C_v}  x_1\cdots x_n |\det M(x)_{a_1,\ldots,a_n}| \, dw \\
    = \int_{C_v}  e^{-\langle b_1+\ldots +b_n, w\rangle} |\det [b_1,\ldots,b_n]| \, dw .
 \end{split}
\end{equation}
By Proposition~\ref{pro:crucial-trick}, this is at most~$1$.
This shows claim \eqref {eq:STS} and 
finishes the proof of Theorem~\ref{th:main-mixed-real}. \qed

\subsection{Proof of Proposition~\ref{pro:MVR}}

For finite $S_i\subseteq\R$, put  $A:=S_1\times\ldots\times S_n$, 
and consider %the maps
\begin{equation}
  \begin{split}
 \psi_i &\colon\R_{>0}\to\proj(\R^{S_i}), x_i \mapsto [x_i^{a_i}]_{a_i \in S_i} ,\\
 \psi &\colon\R^n_{>0}\to\proj(\R^{A}), x \mapsto [x^{a}]_{a\in A}
 \end{split}
\end{equation}
with images $Z_i$ and $Z$, respectively.
The kinematic formula for real projective space \cite[Cor.~A.3]{BL:19} gives 
\begin{equation*}
 \E(S_i) = \frac{\vol(Z_i)}{\vol(\proj^1)}, \quad   
 \E(A,\ldots,A) = \frac{\vol(Z)}{\vol(\proj^n)} .
\end{equation*}
%Moreover, put  $A:=A_1\times\ldots\times A_n$, and 
%consider the map 
%$$
%\psi\colon\R^n_{>0}\to\proj(\R^{A_i}), x \mapsto [x^{a}]_{a\in A} 
%$$
%with image $Z$
%$\psi$ defined in \eqref{eq:Gmap-psi} 
The key insight is that $Z$ is obtained as the image of $Z_1\times\ldots\times Z_n$ 
under the {\em Segre embedding} 
$$
 \proj(\R^{S_1})\times\ldots\times\proj(\R^{S_n}) \to \proj(\R^{S_1}\otimes\ldots\otimes \R^{S_n}) \simeq \proj(\R^{A}) ,
$$
which is isometric (see Appendix~\ref{se:A2}).  
Therefore, we have 
$\vol(Z) = \vol(Z_1)\cdots \vol(Z_n)$, 
which completes the proof of Proposition~\ref{pro:MVR}. \qed

\subsection{Proof of Proposition~\ref{th:main-unmixed-real}}

Given is a finite subset $A\subseteq\R^n$ with convex hull $P$. 
By Lemma~\ref{le:psi-inj} we can can w.l.o.g.\ assume that $\dim P=n$.  
Consider the injective map 
\begin{equation}\label{eq:def-psi}
 \psi \colon \R^n_{>0} \to \proj(\R^{A}), \,  \psi(x) := [x^{a}]_{a\in A} 
\end{equation}
with image $Z\subseteq \proj(\R^{A})$. 
The kinematic formula for real projective space
\cite[Cor.~A.3]{BL:19} 
is considerably simpler than the one in 
Theorem~\ref{th:KF}, since $O(m)$ acts transitively 
on the Grassmann manifolds $\Gr(k,\R^m)$: we have 
\begin{equation}\label{eq:unm-IG}
 \E(A,\ldots,A) = \frac{\vol(Z)}{\vol(\proj^n)} 
 = \frac{1}{\vol(\proj^n)} \int_{\R^n_{>0}} J\psi(x)\, dx .
\end{equation}
We now proceed as in the proof of Theorem~\ref{th:main-mixed-real}.
%Subsection~\ref{se:pf-main-mixed}. 
We make the coordinate change $x=e^{-w}$ and 
decompose the resulting integral according to the decomposition~\eqref{eq:decomp}
of $\R^n$ into the full dimensional cones~$C_v$ corresponding to vertices~$v$.
Thus
\begin{eqnarray*}
  \lefteqn{\int_{\R^n_{>0}} J\psi(x)\, dx = \int_{\R^n} x_1\cdots x_nJ\psi(x)\, dw }\\
 &= \sum _{C_v}  \int_{C_v} x_1\cdots x_nJ\psi(x)\, dw 
\end{eqnarray*}
For a fixed vertex~$v$ of~$P$, 
we consider the map 
$\varphi\colon\R^n_{>0} \to \R^{A\setminus\{v\}}$ defined by 
\begin{equation}\label{def:varphi}
 \varphi(x) = ( x^{a- v})_{a\in A \setminus\{v\}} .
\end{equation}
Then we have $\psi(x) = \pi(\varphi(x))$, 
where $\pi$ is the inverse of the chart 
$\proj(\R^A)_{y_v\ne 0} \to \R^{A\setminus\{v\}}$. 
It is easy to verify that $J\psi(x) \le J\varphi(x)$
using $\|D_{\varphi(x)}\pi\| \le 1$, see~\eqref{eq:der-pi}.

Le us view $\caM(x) := D_x\varphi$ as a matrix 
whose rows are labelled by elements of $A\setminus\{v\}$. 
and denote by $\cM(x)_{a_1,\ldots,a_n}$ the submatrix 
of $\cM(x)$ obtained by selecting the rows 
labelled by the~$a_i$. %$a_1,\ldots,a_n$.  
Cauchy-Binet implies that 
$$
 J\varphi(x)^2 = \det(\cM(x)^T \cM(x)) = 
  \sum_{a_1,\ldots,a_n}  (\det \cM(x)_{a_1,\ldots,a_n})^2, 
$$
with the sum running over all $n$-element subsets $\{a_1,\ldots,a_n\}$ 
of $A\setminus\{v\}$, of which there are ${t-1 \choose n}$ many.
This implies 
$J\varphi(x) \ \le\  \sum_{a_1,\ldots,a_n}  |\det \cM(x)_{a_1,\ldots,a_n} |$. 
We have arrived at 
\begin{eqnarray*}
  \lefteqn{\int_{C_v}  x_1\cdots x_n J\psi(x) \, dw}  \\
  & \le\ \sum_{a_1,\ldots,a_n} 
  \int_{C_v}  x_1\cdots x_n |\det \caM(x)_{a_1,\ldots,a_n}| \, dw  
   \le\ {t-1 \choose n} ,
\end{eqnarray*}
where the right-hand inequality follows from Proposition~\ref{pro:crucial-trick}
as in \eqref{eq:anschluss}. %This completes the proof. 
\qed

\subsection{Proof of Theorem~\ref{th:real_concentrate_roots}} 

The key observation is the following. 
Define for~$\e>0$ 
$$
 D_\e:= \{x\in\R^n \mid \|x\|\ge \e \}.
$$

\begin{lemma}\label{le:lim-eI}
Let $C\subseteq\R^n$ be a proper cone, $d\in\inte(C^*)$, and $\e>0$. Then 
$$
 \lim_{m\to\infty }m^n \int_{C\cap D_\e} e^{-m\langle d,w\rangle} \ dw = 0
$$
\end{lemma}

\begin{proof}
Since $\cap_{m\ge 1} D_{m\e} = \varnothing$, 
basic integration theory implies 
$$
 \lim_{m\to\infty} \int_{C\cap D_{m\e}} e^{-\langle d,u\rangle} \ du = 0.
$$
Making the change of variables $u=mw$ shows the claim. %assertion.
\end{proof}

We now observe the following. Let $U\subseteq\R^n_{>0}$ be open. 
Analogously as for~\eqref{eq:NAI}, one shows that 
$$
 (2\pi)^{-\frac{n}{2}} \rho_{m_1}\cdots \rho_{m_n} 
  \int_{U} x_1\cdots x_n(\overline{\s}_Z\circ \psi) J\psi \, dw .
$$
equals the expected number of nondegenerate zeros {\em in $U$} of the random system~\eqref{eq:def_F_Syst}.

We follow the proof of Theorem~\ref{th:main-mixed-real}. 
Note that stretching the support does not change the Newton polytopes $P_i$ and $P=P_1+\ldots+P_n$.
Fix a vertex $v$ of $P$. According to Lemma~\ref{le:V-sum-polytope},  
there are vertices $v_i$ of~$P_i$, for $i=1,\ldots,n$, satisfying 
$v=v_1+\ldots+v_n$. 
Tracing the proof of Theorem~\ref{th:main-mixed-real}, one sees that 
it is sufficient to show that (compare \eqref{eq:anschluss}) 
for any selection $a_1\in A_1\setminus\{v_1\},\ldots,a_n\in A_n\setminus\{v_n\}$, 
the vectors $b_i= a_i -v_i$ satisfy 
\begin{equation*}
% \int_{C_v}  x_1\cdots x_n |\det M(x)_{j_1,\ldots,j_n}| \, dw = 
 \lim_{m\to\infty } \int_{C_v}  e^{-m\langle b_1+\ldots +b_n, w\rangle} |\det [mb_1,\ldots,mb_n]| \, dw = 0.
\end{equation*}
However, this is a consequence of Lemma~\ref{le:lim-eI}. \qed

%%%
%\bibliographystyle{plain}
%\bibliography{lit.bib}

\begin{acks}
We thank the referees for their comments, in particular for pointing out 
an error in the interpretation of the bound in Proposition~\ref{th:main-unmixed-real}. 
The author is supported by the ERC under the European Union's Horizon 2020 research 
and innovation programme (grant agreement no. 787840).
\end{acks}

%Authors should not prepare this section as a numbered or unnumbered {\verb|\section|}; please use the ``{\verb|acks|}'' environment.

%\section{Appendices}
%
%If your work needs an appendix, add it before the
%``\verb|\end{document}|'' command at the conclusion of your source
%document.
%
%Start the appendix with the ``\verb|appendix|'' command:
%\begin{verbatim}
%  \appendix
%\end{verbatim}
%and note that in the appendix, sections are lettered, not
%numbered. This document has two appendices, demonstrating the section
%and subsection identification method.

%%
%% The next two lines define the bibliography style to be used, and
%% the bibliography file.
%\bibliographystyle{ACM-Reference-Format}
%\bibliography{lit}

%%% -*-BibTeX-*-
%%% Do NOT edit. File created by BibTeX with style
%%% ACM-Reference-Format-Journals [18-Jan-2012].

%%
%% If your work has an appendix, this is the place to put it.
\appendix

\section{Proof of Proposition~\ref{prop:sigma-symm}}\label{se:A0}
%\begin{proof}
Let $\nu_1,\ldots,\nu_m$ be an orthonormal basis of $V^\perp$. 
We decompose $\nu_i= \nu'_i + \nu''_i$ according to $E=W \oplus W^\perp$. 
Then $p(\nu_i)= \nu'_i$ and $|\det p\,| = \|\nu'_1\wedge\ldots\wedge \nu'_m\|$. 
If $\omega_1,\ldots,\omega_k$ denotes an orthonormal basis of $W^\perp$, we have
\begin{eqnarray*}
  \lefteqn{ \sigma(V^\perp,W^\perp) = \|\nu_1\wedge\ldots \wedge \nu_m\wedge \omega_1\wedge\ldots \wedge \omega_k\|} \\
 &= \|\nu'_1\wedge\ldots \wedge \nu'_m\wedge \omega_1\wedge\ldots \wedge \omega_k\|  
 &= \|\nu'_1\wedge\ldots \wedge \nu'_m\| ,
\end{eqnarray*}
the last equality holding since 
the span of the $\nu'_i$ equals $W$, which is 
orthogonal to the span of the~$w_j$, which is $W^\perp$. 
This proves $\sigma(V^\perp,W^\perp) = |\det p\,|$.

For the second assertion, we use that 
$|\det p\,| = |\det q\,|$, where $q\colon W^\perp \to V$ 
denotes the restriction of the orthogonal projection $E\to  V$ to $W^\perp$, 
see \cite[Lemma~5.4]{rigid-II}. \qed
%\end{proof}

\section{Proof of Lemma~\ref{le:partial_volume}}\label{se:A1}
%\begin{proof}[Proof of Lemma~\ref{le:partial_volume}]
We fix $x\in \inte(C^*)$. 
For $t\ge 0$ we define the $n-1$-dimensional slice 
$$
 C_t := \big\{y\in C \mid \langle x, y \rangle = t \|x\| \big\} . 
$$
By Fubini, we get 
\begin{eqnarray*}\label{eq.Fubini}
 v_C(x) &= \int_C e^{-\langle x,y\rangle} \, dy = \int_0^\infty \vol_{n-1}(C_t) e^{-t\|x\|} dt \\
 &= \vol_{n-1}(C_1) \int_0^\infty  t^{n-1} e^{-t\|x\|} dt .
\end{eqnarray*}
%since the volume of $C_t$ scales as 
%$\vol_{n-1}(C_t) = t^{n-1}  \vol_{n-1}(C_1)$. 
Note that 
\begin{eqnarray*}
 \int_0^\infty  t^{n-1} e^{-t\|x\|} dt  = \frac{1}{\|x\|^n} \int_0^\infty  s^{n-1} e^{-s} ds 
   = \frac{(n-1)!}{\|x\|^n} . %= \frac{1}{\|x\|^n} \Gamma(n)
\end{eqnarray*}
Moreover, we have 
\begin{eqnarray*}
 \vol_{n-1}(C_1) &=  n\, \vol_{n}\big\{y\in C \mid \langle x, y \rangle \le \|x| \big\} \\
  &= n\, \|x\|^{n} \, \vol_{n}\big\{y\in C \mid \langle x, y \rangle \le 1 \big\} .
\end{eqnarray*}
It follows that 
\begin{eqnarray*}
 v_C(x) &= n\, \|x\|^{n} \, \vol_{n}\big\{y\in C \mid \langle x, y \rangle \le 1 \big\} \ \frac{(n-1)!}{\|x\|^n} \\
  &= n!\, \vol_{n}\big\{y\in C \mid \langle x, y \rangle \le 1 \big\} ,
\end{eqnarray*}
completing the proof.
\qed
%\end{proof}

\section{The Segre embedding is isometric}\label{se:A2}

Consider the Segre embedding 
$$
 S\colon\proj(\R^{m}) \times\proj(\R^{n}) \to \proj(\R^{m\times n}),\ 
  ([x],[y]) \mapsto [x_i y_j] .  
$$
It is well known that $S$ is a smooth embedding. 
If we endow the real projective space with the standard Riemannian metric 
(see \S~\ref{se:Pproj}), then $S$ is isometric. 
This is also true for the Segre embedding with several factors. 
We provide the proof for lack of reference.

\begin{prop}
The Segre embedding is isometric. 
\end{prop}

\begin{proof}
For notational simplicity, we  restrict ourselves to the case of two factors 
We need to show that the derivatives of $S$ preserve the inner products. 
By orthogonal invariance, it suffices to consider the derivative at 
$([e_0],[e_0])$, which is mapped to $[E_{00}]$. We can 
isometrically identify 
the tangent spaces at these points with 
$\R^{m-1}\times\R^{n-1}$ and $\R^{mn-1}$, respectively.  
Then derivative of $S$ at $([e_0],[e_0])$ is given by 
$$
\R^{m-1}\times\R^{n-1} \to \R^{m\times n}, \,
(v,w) \mapsto 
\begin{bmatrix}
 0   & w^T \\
 v & 0 
\end{bmatrix} .
%\begin{bmatrix}
% 0   & v_1& \ldots & v_{m-1} \\
% w_1 & 0 &  \ldots & 0 \\
% \vdots & 0 & \ddots & 0 \\
% w_{n-1} & 0 &  \ldots & 0 \\
%\end{bmatrix} .
$$
Clearly, this map preserves the inner products.
\end{proof}

%xxx

\section{Supplement}\label{se:A3} 

It is instructive to see how \eqref{eq:unm-IG} directly
follows from the more general kinematic formula in Theorem~\ref{th:KF}.
%To see this, 
Consider the injective map $\psi$ from~\eqref{eq:def-psi} 
with image $Z\subseteq\proj(\R^A)$. 
We use $\psi$ to define the map 
\begin{equation}\label{eq:psi-d}
\psi_d\colon \R^n_{>0} \to (\proj(\R^A))^n,\, x\mapsto (\psi(x),\ldots,\psi(x)).  
\end{equation}
The image 
$Z_d = \{(y,\ldots,y) \mid y \in Z \} \subseteq (\proj^m)^n$ 
of $\psi_d$
%$Z_d = \{(y,\ldots,y) \mid y \in Z \} \subseteq (\proj^m)^n$ 
is the diagonal embedding of $Z$ in the product of 
projective spaces. 
By Theorem~\ref{th:KF} and \eqref{eq:topsy} we have 
$$
 \E(A,\ldots,A) =  (2\pi)^{-\frac{n}{2}} \rho_m^n
  \int_{\R^n_{>0}} (\overline{\s}_{Z_d} \circ \psi_d) J\psi_d\, dx .
$$
Via Lemma~\ref{le:sd} below, we indeed conclude that 
$$
 \E(A,\ldots,A) =    \frac{1}{\vol(\proj^n)} \int_{\R^n_{>0}} J\psi\, dx =
  \frac{\vol(Z)}{\vol(\proj^n)}  ,
$$
which is \eqref{eq:unm-IG}. 

\begin{lemma}\label{le:sd}
For $x\in\R^n_{>0}$ %and $p:=\psi(x)$ 
we have 
$$
\mbox{$\rho_m^{n}\, \overline{\s}_{Z_d}(\psi_d(x)) J\psi_d(x) 
 =   \frac{(2\pi)^{\frac{n}{2}}}{\vol(\proj^n)} J\psi(x)$.} 
$$
\end{lemma}

\begin{proof} 
Lemma~\ref{le:asf} applied to the map $\psi_d$ from \eqref{eq:psi-d} gives 
\begin{eqnarray} \notag
\lefteqn{\rho_{m}^{n}\, \overline{\s}_{Z_d}(\psi_d(x)) J\psi_d(x)} \\ \label{eq:R(x)2}
&=\E_{\lambda_1,\ldots,\lambda_n} 
\left\| (\lambda_1 \circ D_x \psi) \wedge\ldots\wedge (\lambda_n\circ D_x \psi) \right\|
\end{eqnarray}
where the $\lambda_i$ are standard Gaussian linear forms on $T_{\psi(x)}\proj^m$. 
Take an isometry $T_{\psi(x)}\proj^m\simeq\R^m$, 
view $\lambda_i\in\R^m$ as a vector, and view 
$\Delta:=D_x \psi$ as a matrix in $\R^{m\times n}$. 
We note that 
$J\psi(x) = \sqrt{\det (\Delta^T\Delta)}$.
The right-hand side of~\eqref{eq:R(x)2} 
can be written as the expectation 
$\E_{\lambda_i} |\det R(x) |$, 
with the matrix  
\begin{equation}\label{eq:R(x)1}
R(x):= 
\begin{bmatrix}
\lambda_1^T \circ D_x \psi \\
\vdots\\
\lambda_n^T \circ D_x \psi 
\end{bmatrix}
=
\begin{bmatrix}
\lambda_1^T \\
\vdots \\
\lambda_n^T  \\
\end{bmatrix}
\cdot
\Delta .
\end{equation}
%and  standard Gaussian matrix $\Lambda\in\R^{m\times n}$ 
%with rows $\lambda_i^T$. 
We thus need to prove that 
\begin{equation}\label{eq:EWR}
 \mbox{$\E_{\lambda_i} |\det R(x) | = \frac{(2\pi)^{\frac{n}{2}}}{\vol(\proj^n)} \sqrt{\det(\Delta^T\Delta)} $.}
\end{equation}
%where the expectation is over a standard Gaussian matrix $\Lambda$. 

In order to show this, by the singular value decomposition, we may 
assume that 
$\Delta =\begin{bmatrix} D \\ 0\end{bmatrix}$, 
where $D=\diag(\s_1,\ldots,\s_n)$. 
Note that $\sqrt{\det(\Delta^T\Delta)} = \s_1\cdots\s_n$. 
Then \eqref {eq:R(x)1} can be written as 
$R(x)=\Lambda D$, where $\Lambda\in\R^{n\times n}$ 
is a standard Gaussian {\em square} matrix  
and we get 
$\E_\Lambda |\det(R(x))| = \s_1\cdots\s_n\, \E_w |\det \Lambda|$.
It is well known that   
$\E_\Lambda |\det \Lambda| = \rho_n\rho_{n-1}\cdots\rho_1$,
e.g., see \cite[Cor.~4.11]{Condition}. 
Moreover, %We have %On the other hand 
$\rho_n\rho_{n-1} \cdots \rho_1 = \frac{(2\pi)^{\frac{n}{2}}}{\vol(\proj^n)}$,
since
$\rho_m = \sqrt{2\pi} \,\frac{\vol(\proj^{m-1})}{\vol(\proj^m)}$
by~\cite[Lemma~2.25]{Condition}.
%hence 
%$\rho_n\rho_{n-1} \cdots \rho_1 = \frac{(2\pi)^{\frac{n}{2}}}{\vol(\proj^n)}$.
We have thus verified~\eqref{eq:EWR}.
\end{proof}

\end{document}